\documentclass[11pt]{amsart}%
\usepackage{amsfonts}
\usepackage{amsmath}
\usepackage{amssymb}
\usepackage{graphicx}
\usepackage{amssymb}%
\theoremstyle{plain}

\numberwithin{equation}{section}
\newtheorem{thm}{Theorem}[section]
\newtheorem{cor}[thm]{Corollary}
\newtheorem{lem}[thm]{Lemma}

\theoremstyle{definition}

\theoremstyle{remark}

\numberwithin{equation}{section}
\begin{document}
\title[Groups with finiteness conditions on the lower central series]{Groups with finiteness conditions on the lower central series of non-normal subgroups}
\author{Fausto De Mari}
\address{Dipartimento di Ingegneria\\
Universit\`{a} degli Studi di Napoli Parthenope\\
Centro Direzionale Isola C4\\
80143 Napoli, Italy}
\email{fausto.demari@uniparthenope.it}

\begin{abstract}
It is known that any locally graded group with finitely many
derived subgroups of non-normal subgroups is finite-by-abelian.
This result is generalized here, by proving that in a locally
graded group $G$ the subgroup $\gamma_{k}(G)$ is finite if the set
$\{\gamma_{k}(H)\;|\;H\ntriangleleft G\}$ is finite. Moreover,
locally graded groups with finitely many $k$th terms of lower
central series of infinite non-normal subgroups are also
completely described.
\end{abstract}
\maketitle


\section{Introduction}

Restrictions on the derived subgroup of a group can be obtained
through various finiteness conditions. For instance, F. de
Giovanni and D.J.S. Robinson \cite{dGR05} proved that if a locally
graded group $G$ has finitely many derived subgroups, then its
derived subgroup $G^{\prime}$ is finite, and the assumption that
the group $G$ is locally graded cannot be omitted as can be seen
from the consideration of Tarski groups (i.e. infinite simple
groups in which any proper non-trivial subgroup has prime order).
Recall that a group $G$ is said to be \textit{locally graded} if
each finitely generated non-trivial subgroup of $G$ contains a
proper subgroup of finite index; of course, all locally
(soluble-by-finite) groups are locally graded. In \cite{dGR05},
the authors also proved that a locally graded group has finitely
many derived subgroups of infinite subgroups if and only if it is
either finite-by-abelian or an irreducible \v{C}ernikov group.
Here a  \v{C}ernikov group is said to be \textit{irreducible} if
its largest divisible abelian subgroup $D$ is not central and $D$
does not contain infinite proper $K$-invariant subgroups for each
subgroup $K$ of $G$ such that $C_{G}(D)<K$. In \cite{DMdG06}, F.
De Mari and F. de Giovanni proved that a locally graded group is
finite-by-abelian provided it has finitely many derived subgroups
of non-normal subgroups and that a locally graded group having
finitely many derived subgroups of infinite non-normal subgroups
is either a finite-by-abelian group or an irreducible \v{C}ernikov
group. Recently, S. Rinauro \cite{R13} proved that if $G$ is a
locally graded group and for some integer $k\geq 2$ the set
$\Gamma_{k}(G)=\{\gamma_{k}(H)~|~H\leq G\}$ is finite, then the
subgroup $\gamma_{k}(G)$ is finite;
moreover, if $G$ is a locally graded group and the set $\Gamma_{k}^{\infty}%
(G)=\{\gamma_{k}(H)~|~H\leq G,~H$ infinite$\}$ is finite, then
either $\gamma_{k}(G)$ is finite or $G$ is an irreducible
\v{C}ernikov group.

The results quoted here suggest to consider
 groups $G$ for which the set
\[
\bar{\Gamma}_{k}(G)=\{\gamma_{k}(H)~|~H\ntriangleleft G\}
\]
is finite, for a given positive integer
$k\geq2$. We will refer to such a group $G$ as a $\bar{\Gamma}_{k}%
-$\textit{group} and we will prove the following result.

\bigskip\noindent\textbf{Theorem A.} \textit{Let $G$ be a locally graded
$\bar{\Gamma}_{k}-$group. Then $\gamma_{k}(G)$ is finite. }

\bigskip A group $G$ will be called a
$\bar{\Gamma}_{k}^{\infty}%
-$\textit{group} if the set
\[
\bar{\Gamma}_{k}^{\infty}(G)=\{\gamma_{k}(H)~|~H\ntriangleleft
G,H\text{~infinite}\}
\]
is finite (here again $k$ is a positive integer such that
$k\geq2$). For locally graded $\bar{\Gamma}_{k}^{\infty}-$groups
the following result will be obtained.

\bigskip\noindent\textbf{Theorem B.} \textit{Let $G$ be a locally graded
$\bar{\Gamma}_{k}^{\infty}-$group. Then either $G$ is an
irreducible \v{C}ernikov group or $\gamma_{k}(G)$ is finite. }

\bigskip
A group is \textit{metahamiltonian} if all its non-normal
subgroups are abelian. Such groups are involved in the
consideration of groups in which the set of derived subgroups of
non-normal subgroups is finite, since they are precisely the
groups for which such a set does not contain non-trivial
subgroups. In our consideration of $\bar{\Gamma}_{k}-$groups (or
$\bar{\Gamma}_{k}^{\infty}-$groups) information is needed about
groups in which the non-normal (infinite) subgroups are nilpotent
of class at most $k-1$. The behaviour of such groups will be
investigated in section 2 of this paper, while section 3 will be
devoted to the proof of our main theorems.

Most of our notation is standard and can be found in \cite{R72}.

\bigskip

\section{Groups whose non-normal subgroups are nilpotent}

\medskip In this section groups in which each subgroup (respectively,
infinite subgroup) is either normal or nilpotent of class at most
$c$ (where $c$ is a fixed positive integer) are considered, such a
class of groups coincides with that of all groups $G$ for which
the set $\bar{\Gamma}_{c-1}(G)$ (respectively,
$\bar{\Gamma}_{c-1}^{\infty}(G)$) does not contain non-trivial
subgroups. The behaviour of groups with this property represents
the first step in the study of groups satisfying the property
$\bar{\Gamma}_{k}$ or $\bar{\Gamma}_{k}^{\infty}$, and it can also
be seen in relation with \cite{BP81} where (generalized) soluble
groups were considered that are not locally nilpotent but in which
all non-normal subgroups are locally nilpotent.

\medskip

\begin{lem}
\label{SF}Let $G$ be a locally graded group whose infinite
non-normal subgroups are nilpotent. Then $G$ is soluble-by-finite
and locally satisfies the maximal condition.
\end{lem}

\begin{proof}
Let $H$ be any infinite subgroup of $\gamma_{3}(G)$ and assume
that $H$ is not nilpotent. If $K$ is any subgroup of $G$
containing $H$, then $K$ is infinite and non-nilpotent and so it
is a normal subgroup of $G$. Therefore $G/H$ is a Dedekind group
and hence $\gamma_{3}(G)$ is contained in $H$. This proves that
any proper subgroup of $\gamma_{3}(G)$ is either finite or
nilpotent; in particular, $\gamma_{3}(G)$ satisfies the minimal
condition on non-nilpotent subgroups and thus it is
soluble-by-finite (see \cite{DES01}). It follows that $G$ is
soluble-by-finite and so it locally satisfies the maximal
condition (see \cite{PW78}, Theorem A).
\end{proof}

\medskip In our argument we need the following elementary lemma, which is
probably already well-known.

\medskip

\begin{lem}
\label{lntf}Let $G$ be a locally nilpotent torsion-free group and
let $H$ be a subgroup of $G$ which is nilpotent of class at most
$c$. If $H$ has finite index in $G$, then also $G$ is nilpotent of
class at most $c.$
\end{lem}

\begin{proof}
Since the index $|G:H|$ is finite and $G$ is torsion-free, it
follows that $Z_{c}(H)=H\cap Z_{c}(G)$ (see \cite{LR04}, 2.3.9).
Therefore $H=Z_{c}(H)$ is contained in $Z_{c}(G)$ and hence
$G/Z_{c}(G)$ is finite. Thus $\gamma _{c+1}(G)$ is finite (see
\cite{R72} Part 1, p.113) and so even $\gamma _{c+1}(G)=\{1\}$.
\end{proof}

\medskip

\begin{lem}
\label{lemma torsionfree}Let $G$ be a locally nilpotent
torsion-free group whose non-normal subgroups are nilpotent of
class at most $c$. Then $G$ is nilpotent of class at most $c$.
\end{lem}

\begin{proof}
Clearly we may suppose that $G$ is finitely generated. If $c=1$,
then all non-normal subgroups of $G$ are abelian and hence $G$ is
likewise abelian (see \cite{DMdG05}, Theorem 3.4). Let now
$c\geq2.$ Denote by $\mathcal{L}$ the set of all subgroups of
finite index of $G$ and assume first that no subgroup of $G$ which
belongs to $\mathcal{L}$ is nilpotent of class at most $c$. Then,
if $H\in\mathcal{L}$, any subgroup containing $H$ is normal in $G$
and so the factor $G/H$ is a Dedekind group; in particular, $H$
contains $\gamma_{3}(G).$ Since any finitely generated nilpotent
group is residually
finite, it follows that%
\[
\gamma_{3}(G)\leq\bigcap_{H\in\mathcal{L}}H=\{1\}
\]
and hence $G$ is nilpotent of class $2\leq c.$ This contradiction
proves that there is a subgroup in $\mathcal{L}$ which is
nilpotent of class at most $c$, and so $G$ is likewise nilpotent
of class at most $c$ by Lemma \ref{lntf}.
\end{proof}

\medskip We are now in position to prove that locally graded groups whose
non-normal subgroups are nilpotent of bounded class are
finite-by-nilpotent.

\medskip

\begin{thm}
\label{teorema3}Let $G$ be a locally graded group whose non-normal
subgroups are nilpotent of class at most $c$. Then
$\gamma_{c+1}(G)$ is finite.
\end{thm}

\begin{proof}
The group $G$ is soluble-by-finite by Lemma \ref{SF} and so
Theorem B of \cite{BP81} allows us to suppose that $G$ is locally
nilpotent. Clearly we may also suppose that $G$ is not nilpotent
of class $c$, so that there exists a finitely generated subgroup
$E$ of $G$ which is not nilpotent of class $c$. Then $E$ is a
normal subgroup of $G$ and the factor $G/E$ is a Dedekind group;
thus $G^{\prime}$ is a finitely generated nilpotent group. On the
other hand, if $T$ is the subgroup consisting of all elements of
finite order of $G$, then $G/T$ is nilpotent of class at most $c$
by Lemma \ref{lemma torsionfree}, so that $\gamma_{c+1}(G)$ is
periodic and so even finite.
\end{proof}

\medskip The next lemma follows easily from a result of D.I. Zai\v{c}ev
\cite{Z74}, we give its proof here for the convenience of the
reader (see also \cite{DMdG07}, Lemma 2.8).

\medskip

\begin{lem}
\label{lemmino} Let $G$ be a periodic soluble-by-finite group and
let $H$ be a finite subgroup of $G$. If $G$ is not a \v{C}ernikov
group, there exists a collection $(K_{i})_{i\in I}$ of infinite
subgroups of $G$ such that $\displaystyle\bigcap_{i\in I}K_{i}=H$.
\end{lem}

\begin{proof}
Since $H$ is finite, the group $G$ contains an abelian subgroup
$A$ such that $A^{H}=A$ and $A$ does not satisfy the minimal
condition on subgroups (see \cite{Z74}). Then the socle $S$ of $A$
is infinite and clearly the subgroup $HS$ is residually finite, so
that there exists a normal subgroup of finite index $N$ of $HS$
such that $H\cap N=\{1\}$, and a collection $(L_{i})_{i\in I}$ of
normal subgroups of finite index of $HS$ such that each $L_{i}$ is
contained in $N$ and the intersection
 $\displaystyle\bigcap_{i\in I}L_{i}$ is trivial. Therefore each
$HL_{i}$ is infinite and $\displaystyle\bigcap_{i\in I}HL_{i}=H.$
\end{proof}

\medskip

\begin{thm}
\label{th3bis}Let $G$ be a locally graded group whose infinite
non-normal subgroups are nilpotent (of class at most $c$). Then
either $G$ is a \v{C}ernikov group or all non-normal subgroups of
$G$ are nilpotent (of class at most $c$).
\end{thm}

\begin{proof}
The group $G$ is soluble-by-finite and locally satisfies the
maximal condition by Lemma \ref{SF}. Assume that $G$ contains a
finite non-normal subgroup $H$ which is not nilpotent (of class at
most $c$). Let $g$ be any element of infinite order of $G$, then
$\left\langle H,g\right\rangle $ is infinite and
polycyclic-by-finite. If $\mathcal{L}$ is the set of all subgroups
of finite index of $\left\langle H,g\right\rangle $ containing $H$
and $K$ is any element of $\mathcal{L}$, then $K$ is an infinite
subgroup which is not nilpotent (of class at most $c$) and hence
$K$ is a normal subgroup of $G$. Since a well
known result due to Mal'cev yields that%
\[
H=\bigcap_{K\in\mathcal{L}}K,
\]
it follows that $H$ is normal in $G$. This contradiction proves
that $G$ must be periodic. Since $H$ cannot be the intersection of
infinite subgroups, it follows from Lemma \ref{lemmino} that $G$
is a \v{C}ernikov group and the proof is completed.
\end{proof}

\medskip

\begin{cor}
\label{cor3bis}Let $G$ be a locally graded group whose infinite
non-normal subgroups are nilpotent of class at most $c$. Then
either $G$ is a \v{C}ernikov group or $\gamma_{c+1}(G)$ is finite.
\end{cor}

\begin{proof}
This follows immediately from Theorem \ref{th3bis} and Theorem
\ref{teorema3}.
\end{proof}

\bigskip

\section{Proof of the main results}

\medskip In this section Theorem A and Theorem B are proved.
The first two lemmas show that locally graded groups with the
property $\bar{\Gamma} _{k}^{\infty}$ are locally
polycyclic-by-finite.

\medskip

\begin{lem}
\label{Lemma SF}Let $G$ be a locally graded
$\bar{\Gamma}_{k}^{\infty}-$group. Then $G$ is soluble-by-finite.
\end{lem}

\begin{proof}
We argue by induction on the number $t$ of non-trivial subgroups
in the set $\bar{\Gamma}_{k}^{\infty}(G)$. If $t=0$ the group $G$
is soluble-by-finite by Lemma \ref{SF}, so that $t\geq1$ and
consider any infinite non-normal subgroup $H$ of $G$ such that
$L=\gamma_{k}(H)\neq\{1\}$. Clearly
$L^{g}=\gamma_{k}(H^{g})\in\bar{\Gamma}_{k}^{\infty}(G)$ for every
$g\in G$, so that $L$ has finitely many conjugates and hence
$N=N_{G}(L)$ is a subgroup of finite index of $G$. Since
$L\notin\bar{\Gamma}_{k}^{\infty}(L)$, the set
$\bar{\Gamma}_{k}^{\infty}(L)$ contains less than $t$ non-trivial
subgroups, and hence $L$ is soluble-by-finite by induction on $t$.
Then the $\bar{\Gamma }_{k}^{\infty}-$group $N/L$ is likewise
locally graded (see \cite{DMdG06}, Lemma 4) and hence, since also
the set $\bar{\Gamma}_{k}^{\infty}(N/L)$ contains less than $t$
non-trivial subgroups, again induction on $t$ gives that $N/L$ is
soluble-by-finite. Therefore $N$ is soluble-by-finite and so $G$
is likewise soluble-by-finite.
\end{proof}

\medskip

\begin{lem}
\label{lmax}Let $G$ be a locally graded
$\bar{\Gamma}_{k}^{\infty}-$group. Then $G$ locally satisfies the
maximal condition.
\end{lem}

\begin{proof}
The group $G$ is soluble-by-finite by Lemma \ref{Lemma SF} and so,
in order to prove the lemma, it can be supposed that $G$ is a
finitely generated soluble group. Assume that the statement is
false and choose $G$ as a counterexample with minimal derived
length $d$ and such that the set $\{H_{1},\ldots,H_{t}\}$ of all
non-trivial subgroups which belong to
$\bar{\Gamma}_{k}^{\infty}(G)$ has smallest order $t$; note that
$t\geq1$ by Lemma \ref{SF}. Observe further that $d\neq1$ and we
cannot have $d>2$, otherwise $G^{(d-2)}$ and $G/G^{(d-2)}$ would
both be polycyclic, by the minimal choice of $d$, and so $G$ would
be likewise polycyclic; therefore $G$ is metabelian. Since each
$H_{i}$ has finitely many conjugates, the subgroup $N_{G}(H_{i})$
has finite index in $G$ and hence also the subgroup
$N=N_{G}(H_{1})\cap\ldots\cap N_{G}(H_{t})$ has likewise finite
index in $G$; in particular, $N$ is not polycylic, moreover $N$
contains each $H_{i}$ by the minimal choice of $t$. Since
$\bar{\Gamma}_{k}^{\infty}(N/H_{i})$ contains less than $t$
non-trivial subgroups, the factor $N/H_{i}$ is polycyclic and
hence, if $H=H_{1}\cap\ldots\cap H_{t}$, also $N/H$ is polycyclic;
in particular, $H\neq\{1\}$. Let $x$ be any non-trivial element of
$H$. Since $G$ is residually finite (see \cite{R72} Part 2,
Theorem 9.51), there exists a subgroup of finite index $K$ of $G$
such that $x\notin K$. If $X$ is any infinite subgroup of $K$,
then $\gamma_{k}(X)\neq H_{i}$ for all $i=1,\ldots,t,$ and so
either $\gamma_{k}(X)=\{1\}$ or $X$ is a normal subgroup of $G$;
therefore all infinite non-normal subgroups of $K$ are nilpotent
and thus the finitely generated subgroup $K$ is polycyclic by
Lemma \ref{SF}. It follows that $G$ is polycyclic and this
contradiction concludes the proof.
\end{proof}

\medskip
In what follows some lemmas are given in order to prove that the
$k$th term of the lower central series of any locally graded
$\bar{\Gamma}_{k}^{\infty}-$group is periodic.

\medskip

\begin{lem}
\label{lemma engel}Let $G$ be a $\bar{\Gamma}_{k}^{\infty}-$group
and let $A$ be a finitely generated abelian normal subgroup of
$G$. If $A$ is torsion-free and $g\in G$, then
$\gamma_{k}(\left\langle A,g\right\rangle )=\{1\}$.
\end{lem}

\begin{proof}
Assume first that $A\cap\left\langle g\right\rangle =\{1\}$. If
$\left\langle g\right\rangle $ were a normal subgroup of $G$,
$[A,g]$ would be contained in $A\cap\left\langle g\right\rangle
=\{1\}$ and so $[A,g]=\{1\}$; thus we may suppose that
$\left\langle g\right\rangle $ is not a normal subgroup of $G$.
Since for every infinite subset $I$ of $\mathbb{N}$ we have
\[
\bigcap_{n\in I}\left\langle A^{n},g\right\rangle =\left\langle
g\right\rangle
\text{,}%
\]
it follows that the subgroup $\left\langle A^{n},g\right\rangle $
is normal in $G$ only for finitely many positive integers $n$ and
hence, since the set $\bar{\Gamma}_{k}^{\infty}(G)$ is finite,
there is a positive integer
$\ell$ such that%
\[
\gamma_{k}\left(  \left\langle A^{_{\ell!}},g\right\rangle \right)
=\gamma_{k}\left(  \left\langle A^{_{(\ell+1)!}},g\right\rangle
\right) =\gamma_{k}\left(  \left\langle
A^{_{(\ell+2)!}},g\right\rangle \right) =\ldots
\]
Therefore $\gamma_{k}\left(  \left\langle
A^{_{\ell!}},g\right\rangle \right) =\gamma_{k}\left( \left\langle
A,g\right\rangle \right)  ^{\ell!}$ is a divisible subgroup of
$\left\langle A,g\right\rangle $. But $\left\langle
A,g\right\rangle $ is finitely generated and metabelian, so that
it is residually finite (see \cite{R72} Part 2, Theorem 9.51) and
hence $\gamma _{k}\left(  \left\langle A,g\right\rangle \right)
^{\ell!}=\{1\}.$ Since $\gamma_{k}\left(  \left\langle
A,g\right\rangle \right)  $ is contained in $A$, which is
torsion-free, it follows that $\gamma_{k}\left(  \left\langle
A,g\right\rangle \right)  =\{1\}$.

In the general case, let
\[
A\cap\left\langle g\right\rangle =\left\langle g^{m}\right\rangle
\neq\{1\}
\]
and put%
\[
\frac{A}{A\cap\left\langle g\right\rangle
}=\frac{E}{A\cap\left\langle
g\right\rangle }\times\frac{F}{A\cap\left\langle g\right\rangle }%
\]
where $E/A\cap\left\langle g\right\rangle $ is finite and
$F/A\cap\left\langle g\right\rangle $ is torsion-free. Clearly,
$A\cap\left\langle g\right\rangle $ is contained in
$Z(\left\langle A,g\right\rangle )$ and the factor group
$\left\langle E,g\right\rangle /A\cap\left\langle g\right\rangle $
is finite, so that $\left\langle E,g\right\rangle $ is
central-by-finite. Therefore $\left[  E,g\right]  $ is a finite
subgroup of $A$ (see \cite{R72} Part 1, Theorem~4.12) and so
$\left[E,g\right] =\{1\}$. Since the factor $A/E$ is a finitely
generated abelian torsion-free normal subgroup of $\left\langle
A,g\right\rangle /E$ and $\left\langle gE\right\rangle \cap
A/E=\{1\}$, the first part of this proof yields that
$\gamma_{k}\left( \left\langle A,g\right\rangle /E\right)=~\{1\}.$
Therefore $\gamma_{k}(\left\langle A,g\right\rangle )\leq E\leq
Z\left( \left\langle A,g\right\rangle \right)$ and so
\[
\gamma_{k}(\left\langle A,g\right\rangle )^{m}=\left[  \gamma_{k-1}%
(\left\langle A,g\right\rangle ),\left\langle g\right\rangle
\right] ^{m}=\left[  \gamma_{k-1}(\left\langle A,g\right\rangle
),\left\langle g^{m}\right\rangle \right]  .
\]
Since $g^{m}\in Z(\left\langle A,g\right\rangle )$, it follows
that $\gamma _{k}(\left\langle A,g\right\rangle )^{m}=~\{1\}$,
thus $ \gamma_{k}(\left\langle A,g\right\rangle )=~\{1\} $ because
$\gamma_{k}\left(  \left\langle A,g\right\rangle \right) $ is
contained in $A$ which is torsion-free.
\end{proof}

\medskip

\begin{lem}
\label{lemmahyper}Let $G$ be a finitely generated locally graded
$\bar{\Gamma }_{k}^{\infty}-$group and let $A$ be a torsion-free
abelian normal subgroup of $G$. Then $A$ is contained in the
hypercentre of $G$.
\end{lem}

\begin{proof}
The group $G$ is polycyclic-by finite by Lemma \ref{Lemma SF} and
Lemma
\ref{lmax}, so that Lemma \ref{lemma engel} yields that $\gamma_{k}%
(\left\langle A,g\right\rangle )$ is trivial for any $g\in G$. In
particular, each element $a$ of $A$ is such that $[a,_{k-1~}g]=1$
for all $g\in G$, so that $A$ is contained in the hypercentre of
$G$ (see \cite{R72} Part 2, Theorem 7.21).
\end{proof}

\medskip

\begin{lem}
\label{lemmanil}Let $G$ be a finitely generated locally graded
$\bar{\Gamma }_{k}^{\infty}-$group. If $G$ has no non-trivial
periodic normal subgroups, then $G$ is nilpotent.
\end{lem}

\begin{proof}
The group $G$ contains a polycyclic normal subgroup of finite
index by Lemma \ref{Lemma SF} and Lemma \ref{lmax}, so that the
upper central series of any section of $G$ becomes stationary
after a finite number of steps and thus, in particular, the
hypercentre $\bar{Z}(G)$ of $G$ coincides with $Z_{n}(G)$ for some
a positive integer $n$. Let $T/Z_{n}(G)$ be any periodic normal
subgroup of $G/Z_{n}(G)$. Then $T/Z_{n}(G)$ is finite, so that
$\gamma_{n+1}(T)$ is finite (see \cite{R72} Part 1, p.113) and
hence $\gamma_{n+1}(T)=~\{1\}$ because $G$ does not contain
non-trivial periodic normal subgroups; in particular, $T$ is
nilpotent and torsion-free. Since $T/Z(T)$ is torsion-free (see
\cite{R72} Part~1, Theorem~2.25), it follows from
Lemma~\ref{lemmahyper} and by induction on the nilpotent class of
$T$, that $Z(T)$ is contained in $Z_{n}(G)$ and that $T/Z(T)$ is
contained in the hypercentre of $G/Z(T)$. Therefore $T$ is
contained in $\bar {Z}(G)=Z_{n}(G)$ and so $G/Z_{n}(G)$ has no
non-trivial periodic normal subgroups. Let $K/Z_{n}(G)$ be a
polycyclic normal subgroup of finite index of $G/Z_{n}(G)$, and
let $A/Z_{n}(G)$ be the smallest term of the derived series of
$K/Z_{n}(G)$. If $A/Z_{n}(G)$ were not trivial, it would be a
torsion-free abelian normal subgroup of $G/Z_{n}(G)$ and so it
would be contained in the hypercentre of $G/Z_{n}(G)$ by
Lemma~\ref{lemmahyper}. Therefore $A/Z_{n}(G)$ is trivial, so that
$K/Z_{n}(G)$ is trivial and hence $G/Z_{n}(G)$ is finite. It
follows that $\gamma_{n+1}(G)$ is finite (see \cite{R72} Part 1,
p.113).
 Since $G$ has no non-trivial periodic normal subgroups, we obtain that
$\gamma_{n+1}(G)=\{1\}$ and the proof is completed.
\end{proof}

\medskip

\begin{lem}
\label{lemmanil1}Let $G$ be a finitely generated nilpotent $\bar{\Gamma}%
_{k}^{\infty}-$group. If $G$ is torsion-free, then
$\gamma_{k}(G)=\{1\}$.
\end{lem}

\begin{proof}
By way of contradiction, assume that the statement is false and
among all the
counterexamples choose $G$ in such a way that the set $\bar{\Gamma}%
_{k}^{\infty}(G)$ contains the smallest number $t$ of non-trivial
subgroups. Then $t>0$ by Lemma \ref{lemma torsionfree} and hence
there exists a non-normal subgroup $H$ of $G$ such that
$\gamma_{k}(H)\ $contains a non-trivial element $x$. Since $G$ is
residually finite, there exists a subgroup of finite index $L$ of
$G$ such that $x\notin L.$ If $X$ is any non-normal subgroup of
$L$, then $X$ is not normal in $G$ and $\gamma
_{k}(X)\neq\gamma_{k}(H)$, so that the set
$\bar{\Gamma}_{k}^{\infty}(L)$ contains less than $t$ non-trivial
subgroups and hence $\gamma_{k}(L)=\{1\}$. Therefore Lemma
\ref{lntf} yields that also $\gamma_{k}(G)=\{1\}$ and this
contradiction completes the proof.
\end{proof}

\medskip

\begin{lem}
\label{lemmaperiodic}Let $G$ be a locally graded $\bar{\Gamma}_{k}^{\infty}%
-$group. Then $\gamma_{k}(G)$ is periodic.
\end{lem}

\begin{proof}
The group $G$ is soluble-by-finite by Lemma \ref{Lemma SF} and it
locally satisfies the maximal condition by Lemma \ref{lmax}. Let
$x$ and $y$ be elements of finite order of $G$ and let $X$ be the
largest periodic normal subgroup of $\left\langle x,y\right\rangle
$. Application of Lemma~\ref{lemmanil} yields that the factor
group $\left\langle x,y\right\rangle/X$ must be trivial, so that
$\left\langle x,y\right\rangle=X$ is finite. It follows that the
set $T$ of all elements of finite order of $G$ is a (normal)
subgroup. By Lemma \ref{lemmanil} and Lemma \ref{lemmanil1} each
finitely generated subgroup of $G/T$ is nilpotent of class at most
$k-1$, so that $G/T$ is likewise nilpotent of class at most $k-1$.
Thus $\gamma_{k}(G)$ is contained in $T$ and hence $\gamma_{k}(G)$
is periodic.
\end{proof}

\medskip We are now able to prove our first main result.

\medskip

\begin{proof}
[Proof of Theorem A]Clearly it can be assumed that $\gamma_{k}(G)$
is not trivial and hence $G$ contains a finitely generated
subgroup $E$ such that $\gamma_{k}(E)\neq\{1\}$; moreover, by
Theorem \ref{teorema3} it can be assumed that the set
$\bar{\Gamma}_{k}^{\infty}(G)$ contains $t\geq1$ non-trivial
subgroups. The group $G$ is soluble-by-finite by Lemma \ref{Lemma
SF} and it locally satisfies the maximal condition by Lemma
\ref{lmax}, moreover $\gamma_{k}(G)$ is periodic by Lemma
\ref{lemmaperiodic}. Suppose that $E$ is contained in a finitely
generated non-normal subgroup $F$
of $G$. Then the finite subgroup $\gamma_{k}(F)$ belongs to $\bar{\Gamma}%
_{k}(G)$ so that, since $\bar{\Gamma}_{k}(G)$ is finite,
$\gamma_{k}(F)$ has finitely many conjugates and hence Dietzmann's
lemma (see \cite{R72} Part 1, p.45) yields that the normal closure
$N$ of $\gamma_{k}(F)$ in $G$ is finite. On the other hand, by
induction on $t$ it follows that the subgroup $\gamma _{k}(G/N)$
is finite, and hence also $\gamma_{k}(G)$ is finite. Assume now
that every finitely generated subgroup containing $E$ is normal.
Then $E$ is normal and $G/E$ is a Dedekind group, so that
$G^{\prime }E/E$ is finite and hence $G^{\prime}$ is finitely
generated. Then $G^{\prime }$ is polycyclic-by-finite and hence
its periodic subgroup $\gamma_{k}(G)$ is finite.
\end{proof}

\medskip

\begin{lem}
\label{dir}Let $G$ be a group and let $A$ be an abelian normal
subgroup of finite index of $G$. If $A$ is the direct product of
infinitely many subgroups of prime order, then there exists a
collection $(B_{n})_{n\in\mathbb{N}}$ of finite $G$-invariant
subgroups of $A$ such that
$\displaystyle\left\langle B_{n}~|~n\in\mathbb{N}\right\rangle =\underset{n\in\mathbb{N}%
}{\text{\emph{Dr}}}B_{n}$.

\end{lem}

\begin{proof}
Clearly each subgroup of $A$ has finitely many conjugates in $G$
so that, if $a_{1}$ is any non-trivial element of $A$, the
subgroup $B_{1}=\left\langle a_{1}\right\rangle ^{G}$ is a finite
$G$-invariant subgroup of $A$. Suppose, by induction, that finite
$G$-invariant subgroups $B_{1},\ldots,B_{n}$ of $A$
have been chosen in such a way that%
\[
\left\langle B_{1},\ldots,B_{n}\right\rangle
=B_{1}\times\cdots\times
B_{n}\text{.}%
\]
Since $\left\langle B_{1},\ldots,B_{n}\right\rangle $ is finite
and the group $G$ is residually finite, there exists a normal
subgroup of finite index $N$ of $G$ such that
\[
N\cap\left\langle B_{1},\ldots,B_{n}\right\rangle =\{1\}.
\]
Then $N\cap A$ has finite index in $A$ and hence $N\cap A$
contains a non-trivial element $a_{n+1}$. Thus
$B_{n+1}=\left\langle a_{n+1}\right\rangle ^{G}$ is a finite
$G$-invariant subgroup of $N\cap A$ and
\[
\left\langle B_{1},\ldots,B_{n},B_{n+1}\right\rangle
=B_{1}\times\cdots\times
B_{n}\times B_{n+1}\text{,}%
\]
so that the lemma is proved.
\end{proof}

\medskip

\begin{lem}
\label{periodic}Let $G$ be a locally graded periodic $\bar{\Gamma}_{k}%
^{\infty}-$group. Then either $G$ is a \v{C}ernikov group or
$\gamma_{k}(G)$ is finite.
\end{lem}

\begin{proof}
Assume that the statement is false and let $G$ be a counterexample
in which the set $\bar{\Gamma}_{k}(G)$ has the smallest number $t$
of non-trivial subgroups. Then $t\geq1$ by Corollary
\ref{cor3bis}; moreover, Theorem A allows us to suppose that there
exists a finite non-normal subgroup $H$ of $G$ such that
$\gamma_{k}(H)\neq\{1\}$. Since the group $G$ is soluble-by-finite
by Lemma \ref{Lemma SF} and it is not a \v{C}ernikov group, $G$
contains an abelian subgroup $A$ such that $A^{H}=A$ and $A$ does
not satisfy the minimal condition on subgroups (see \cite{Z74}).
Then the socle of $A$ is infinite and hence, by replacing $A$ by
its socle, we may suppose that $A$ is the direct product of
infinitely many cyclic groups of prime order. Put $K=AH$. Since
$H$ is finite, also $A\cap H$ is finite and then, by replacing $A$
by a suitable $K$-invariant subgroup of finite index, we may
further suppose that $A\cap H=\{1\}$. Application of Lemma
\ref{dir} yields that there exists a collection
$(B_{n})_{n\in\mathbb{N}}$ of finite $K$-invariant
subgroups of $A$ such that%
\[
\left\langle B_{n}~|~n\in\mathbb{N}\right\rangle =\underset{n\in\mathbb{N}%
}{\text{Dr}}B_{n}.
\]
Let%
\[
\hat{X}_{1}=\underset{n\in\mathbb{N}}{\text{Dr}}B_{2n}\text{ \ \ \
\ and
\ \ \ \ }\check{X}_{1}=\underset{n\in\mathbb{N}}{\text{Dr}}B_{2n+1}\text{.}%
\]
Then $\hat{X}_{1}$ and $\check{X}_{1}$ are infinite normal
subgroups of $K$ and $\hat{X}_{1}\cap\check{X}_{1}=\{1\}$; so that
$H=H\hat{X}_{1}\cap H\check{X}_{1}$. Since $H$ is not normal in
$G$, there exists a subgroup $L_{1}$ in
$\{\hat{X}_{1},\check{X}_{1}\}$ such that $HL_{1}$ is not normal
in $G$.

Consider now
$C\in\{\hat{X}_{1},\check{X}_{1}\}\smallsetminus\{L_{1}\}$;
clearly $C$ can be written as
\[
C=\underset{n\in\mathbb{N}}{\text{Dr}}C_{n}%
\]
where each $C_{n}$ is a finite $K$-invariant subgroup of $A$ and
$\left\langle
C,L_{1}\right\rangle =C\times L_{1}$. Let%
\[
\hat{X}_{2}=\underset{n\in\mathbb{N}}{\text{Dr}}C_{2n}\text{ \ \ \
\ and
\ \ \ }\check{X}_{2}=\underset{n\in\mathbb{N}}{\text{Dr}}C_{2n+1}\text{.}%
\]
Then $\hat{X}_{2}$ and $\check{X}_{2}$ are infinite normal
subgroups of $K$, $\hat{X}_{2}\cap\check{X}_{2}=\{1\}$ and
$H=H\hat{X}_{2}\cap H\check{X}_{2}$. Since $H$ in not normal in
$G$, there exists an element $L_{2}$ of $\{\hat {X}_{2},$
$\check{X}_{2}\}$ such that $HL_{2}$ is not normal in $G$.
Iterating this argument it is clear that $t+1$ infinite
$K$-invariant subgroups $L_{1},\ldots,L_{t},L_{t+1}$ can be chosen
in such a way that $L_{i}\cap L_{j}=\{1\}$, $HL_{i}\cap HL_{j}=H$
and $HL_{i}\ $is not normal in $G$ for each
$i,j\in\{1,\ldots,t+1\}$ with $i\neq j$. Since $\gamma_{k}(H)$ is
not trivial, each $\gamma_{k}(HL_{i})$ is likewise non-trivial and
hence,
since there are only $t$ non-trivial subgroups in the set $\bar{\Gamma}%
_{k}^{\infty}(G)$, there exist $\ell,m\in\{1,\ldots,t+1\}$, with
$\ell\neq m$, such that
$\gamma_{k}(HL_{\ell})=\gamma_{k}(HL_{m})$. But $HL_{\ell}\cap
HL_{m}=H$, so that $\gamma_{k}(HL_{\ell})=\gamma_{k}(HL_{m})\leq
H$ and hence $\gamma_{k}(HL_{\ell})$ is a finite subgroup which
belongs to $\bar{\Gamma}_{k}^{\infty}(G).$ Since the set
$\bar{\Gamma}_{k}^{\infty}(G)$ is finite, the finite subgroup
$\gamma_{k}(HL_{\ell})$ has finitely many conjugates and hence
Dietzmann's lemma (see \cite{R72} Part 1, p.45) yields that the
normal closure $N$ of $\gamma_{k}(HL_{\ell})$ in $G$ is finite.
Since $\bar{\Gamma}_{k}^{\infty}(G/N)$ contains less than\ $t$
non-trivial subgroups, it follows\ from the minimal choice of $t$
that $\gamma_{k}(G/N)$ is finite. Thus $\gamma_{k}(G)$ is finite
and this contradiction concludes the proof.
\end{proof}

\medskip

\begin{lem}
\label{cernikov}Let $G$ be a \v{C}ernikov
$\bar{\Gamma}_{k}^{\infty}-$group. If $\gamma_{k}(G)$ is infinite,
then $G$ is irreducible.
\end{lem}

\begin{proof}
Let $D$ be the largest divisible subgroup of $G$, and let $K$ be a
subgroup of $G$ containing $C_{G}(D)$. Assume that $D$ contains an
infinite proper $K$-invariant subgroup $A$. Then the set
$\bar{\Gamma}_{k}(K/A)$ is finite and so Theorem A yields that the
factor group $\gamma_{k}(K)A/A$ is finite. On the other hand,
nilpotent groups satisfying the minimal condition are
finite-by-abelian (see \cite{R72} Part 1, Theorem 3.14 and Theorem
4.12), so that the derived subgroup of $K/\gamma_{k}(K)A$ is
finite and hence $K^{\prime}A/A$ is likewise finite. As the
subgroup $[D,K]$ is divisible, it follows that $[D,K]\leq A<D$ and
hence $C_{D}(K)$ is infinite (see \cite{LR80}, Theorem G). In
particular, $Z(K)$ is infinite and Theorem A yields that
$\gamma_{k}(K)Z(K)/Z(K)$ is finite. Since the nilpotent group
$K/\gamma_{k}(K)Z(K)$ is finite-by-abelian, it follows that the
factor group $K^{\prime}Z(K)/Z(K)$ is finite. Then $[D,K]$ is
contained in $Z(K)$, so that $[D,K,K]=\{1\}$ and hence $D\leq
Z_{2}(K)$. Therefore $K/Z_{2}(K)$ is finite, so that $K$ is
finite-by-nilpotent and so even finite-by-abelian. It follows that
$[D,K]=\{1\}$, so that $K=C_{G}(D)$ and $G$ is an irreducible
\v{C}ernikov group.
\end{proof}

\medskip

\begin{proof}
[Proof of Theorem B]Let $G$ be a counterexample with smallest
number $t$ of non-trivial subgroups which belongs to
$\bar{\Gamma}_{k}^{\infty}(G)$. Then $t\geq1$ by
Corollary~\ref{cor3bis}, while Lemma \ref{periodic} and Lemma
\ref{cernikov} yield that $G$ is not periodic. The group $G$ is
soluble-by-finite by Lemma \ref{Lemma SF}, it locally satisfies
the maximal condition by Lemma \ref{lmax} and $\gamma_{k}(G)$ is
periodic by Lemma \ref{lemmaperiodic}. Since $\gamma_{k}(G)$ is
not trivial, there exists a finitely generated subgroup $E$ of $G$
such that $\gamma_{k}(E)\neq~\{1\}$ and, since $G$ is not
periodic, it can be assumed that $E$ contains some element of
infinite order. In particular, $E$ is an infinite
polycyclic-by-finite group and $\gamma_{k}(E)$ is finite. Assume
that $E$ is not normal in $G$, so that $\gamma_{k}(E)$ belongs to
$\bar{\Gamma}_{k}^{\infty}(G)$ and hence $\gamma_{k}(E)$ has
finitely many conjugates. Since $\gamma_{k}(E)$ is finite, it
follows from Dietzmann's lemma (see \cite{R72} Part 1, p.~45) that
the normal closure $N$ of $\gamma_{k}(E)$ in $G$ is finite. Since
the number of non-trivial subgroups in the set
$\bar{\Gamma}_{k}^{\infty}(G/N)$ is less than $t$, by the minimal
choice of $t$ it follows that $\gamma_{k}(G/N)$ finite. Therefore
$\gamma_{k}(G)$ is likewise finite and this contradiction proves
that $E$ must be a normal subgroup of $G$. Then the factor $G/E$
is a $\bar{\Gamma}_{k}-$group and so application of Theorem A
yields that $\gamma_{k}(G/E)$ is finite. It follows that
$\gamma_{k}(G)$ is finitely generated and so even finite. This
final contradiction proves the theorem.
\end{proof}


\end{document}